\documentclass[12pt,reqno]{amsart}
\usepackage{amssymb}
\usepackage{amsfonts}
\usepackage{graphicx}
\usepackage{color}
\usepackage{hyperref}
\usepackage{lscape}
\usepackage{multirow}
\usepackage{afterpage}

\newtheorem{theorem}{Theorem}[section]
\newtheorem{proposition}[theorem]{Proposition}
\newtheorem{definition}[theorem]{Definition}
\newtheorem{lemma}[theorem]{Lemma}
\newtheorem{corollary}[theorem]{Corollary}
\newtheorem{example}[theorem]{Example}

\def\F{\mathbb{F}}
\def\Z{\mathbb{Z}}
\def\P{\mathcal{P}}
\def\C{\mathcal{C}}
\def\H{\mathcal{H}}
\def\D{\mathcal{D}}
\DeclareMathOperator{\dev}{dev}
\DeclareMathOperator{\Aut}{Aut}
\DeclareMathOperator{\Atop}{Atop}
\DeclareMathOperator{\Apar}{Apar}

\begin{document}

\title{On higher-dimensional symmetric designs}

\author[V.~Kr\v{c}adinac and M.~O.~Pav\v{c}evi\'{c}]{Vedran Kr\v{c}adinac$^1$ and Mario Osvin Pav\v{c}evi\'{c}$^2$}

\address{$^1$Faculty of Science, University of Zagreb, Bijeni\v{c}ka cesta~30, HR-10000 Zagreb, Croatia}

\address{$^2$Faculty of Electrical Engineering and Computing, University of Zagreb,
Unska~3, HR-10000 Zagreb, Croatia}

\email{vedran.krcadinac@math.hr}
\email{mario.pavcevic@fer.hr}

\thanks{This work has been supported by the Croatian Science Foundation
under the project $9752$.}



\date{June 10, 2025}

\begin{abstract}
We study two kinds of generalizations of symmetric block designs
to higher dimensions, the so-called $\C$-cubes and $\P$-cubes.
For small parameters, all examples up to equivalence are determined
by computer calculations. Known properties of automorphisms of
symmetric designs are extended to autotopies of $\P$-cubes, while
counterexamples are found for $\C$-cubes. An algorithm for the
classification of $\P$-cubes with prescribed autotopy groups is
developed and used to construct more examples. A bound on the
dimension of difference sets for $\P$-cubes is proved and
shown to be tight in elementary abelian groups. The construction
is generalized to arbitrary groups by introducing regular sets
of (anti)automorphisms.
\end{abstract}

\maketitle

\section{Introduction}

Two different generalizations of symmetric block designs to higher
dimensions have recently been studied in~\cite{KPT24} and~\cite{KR24}.
A symmetric $(v,k,\lambda)$ design can be represented by its incidence
matrix, i.e.\ by a $v\times v$ matrix~$A$ with $\{0,1\}$-entries satisfying
\begin{equation}\label{eqsbibd}
A A^t = (k-\lambda)I + \lambda J.
\end{equation}
Here, $I$ is the identity matrix and $J$ is the all-ones matrix.
An \emph{$n$-cube of $(v,k,\lambda)$ designs}~\cite{KPT24} is an $n$-dimensional
$v\times \cdots \times v$ matrix such that its every $2$-section is an incidence
matrix of a $(v,k,\lambda)$ design. Sections of dimension~$2$ are submatrices
obtained by fixing all but two coordinates. This generalization is a special
case of W.~de~Launey's proper $n$-dimensional transposable designs; see
\cite[Definitions 2.1, 2.6, 2.7, and Example 2.2]{dL90}. The set of all
$n$-cubes of $(v,k,\lambda)$ designs was denoted by $\C^n(v,k,\lambda)$
in~\cite{KPT24}. We shall refer to them as $\C$-cubes.

The second generalization was introduced in~\cite{KR24} under the name
\emph{$(v,k,\lambda)$ projection $n$-cubes}. These are $n$-dimensional
$v\times \cdots \times v$ matrices with $\{0,1\}$-entries such that
every $2$-projection is an incidence matrix of a $(v,k,\lambda)$ design.
If~$C$ is an $n$-dimensional matrix and $1\le x<y\le n$, the projection
$\Pi_{xy}(C)$ is the $2$-dimensional matrix with $(i_x,i_y)$-entry
\begin{equation*}\label{projsum}
\sum_{1\le i_1,\ldots,i_{x-1},i_{x+1},\ldots,i_{y-1},i_{y+1},\ldots,i_n\le v} C(i_1,\ldots,i_n).
\end{equation*}
The sum is taken over all $n$-tuples $(i_1,\ldots,i_n)\in \{1,\ldots,v\}^n$
with fixed coordinates $i_x$ and $i_y$ in a field of characteristic~$0$. This
definition was inspired by Room squares~\cite{JHD07}, which are generalized
to $n$-dimensional Room cubes in an analogous way. In~\cite{KR24}, the set of
all $(v,k,\lambda)$ projection $n$-cubes was denoted by $\P^n(v,k,\lambda)$.
We shall refer to them as $\P$-cubes. Two examples are shown in Figure~\ref{fig0},
with three light-sources placed along the coordinate axes so that the
projections appear as shadows. These and other pictures in the paper were
rendered using the ray tracing software POV-Ray~\cite{POVRay}.

\begin{figure}[t]
\begin{center}
\includegraphics[width=127mm]{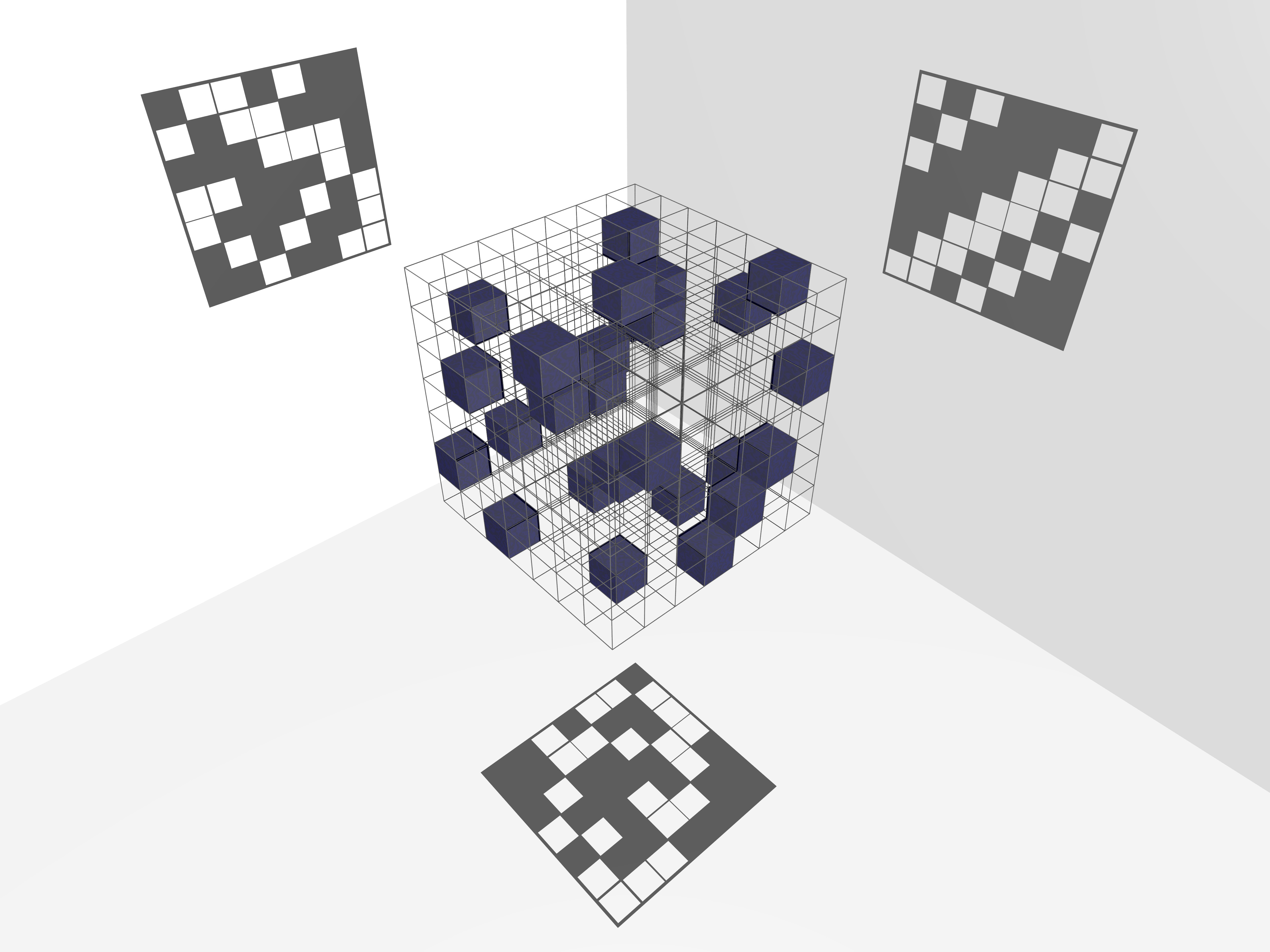}\\[10mm]
\includegraphics[width=127mm]{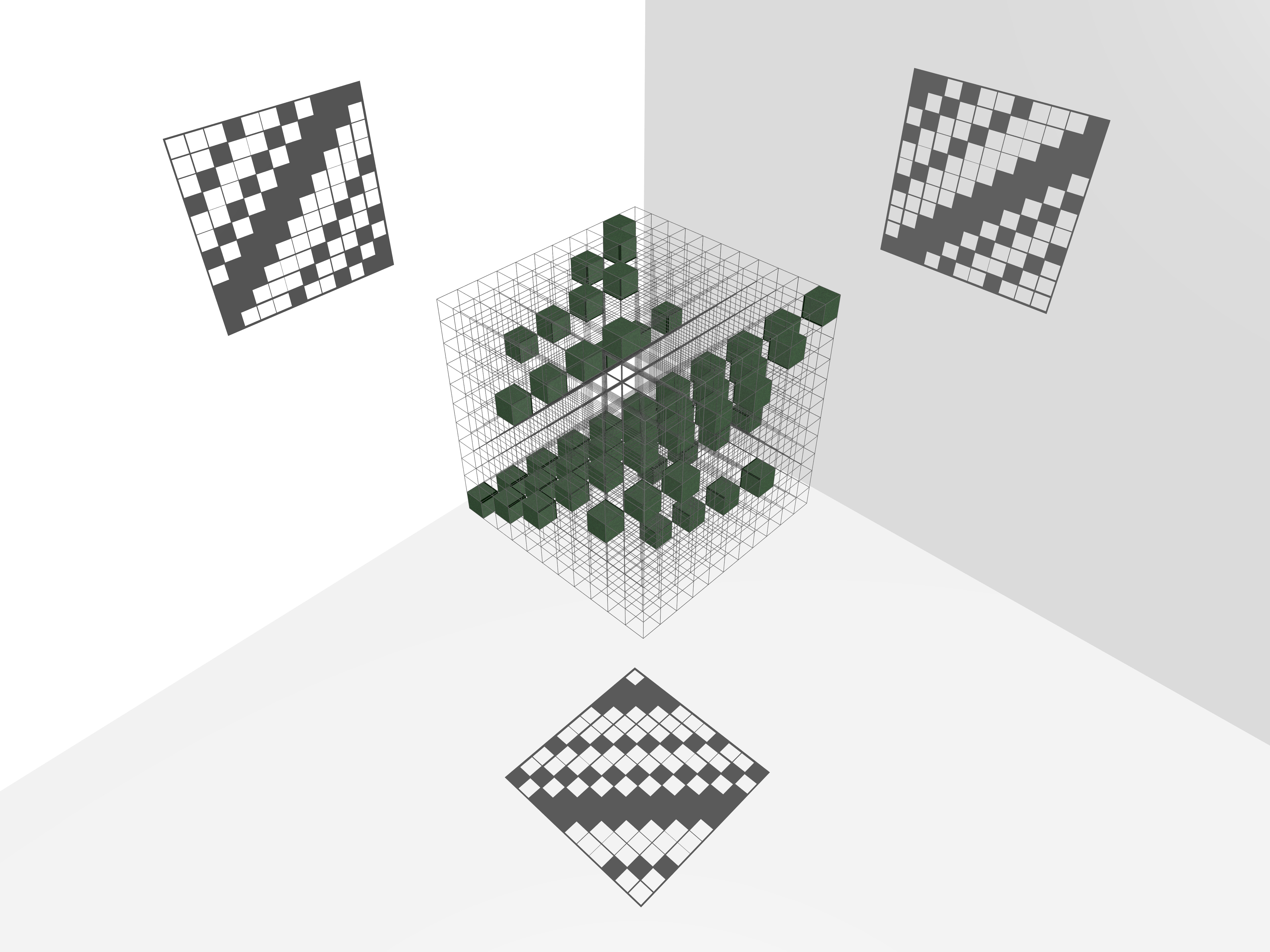}
\end{center}
\caption{Cubes from $\P^3(7,4,2)$ and $\P^3(11,5,2)$.}\label{fig0}
\end{figure}

\afterpage{\clearpage}

The purpose of this paper is to study both types of cubes and to compare their
properties. For dimension $n=2$, both are just incidence matrices of symmetric
designs, but for $n\ge 3$ there are significant differences. For example, it
was shown in~\cite[Theorem 2.9]{KR24} that the dimension of $(v,k,\lambda)$
projection cubes with $k\ge 2$ is bounded by
\begin{equation}\label{dimbound}
n\le \frac{vk-1}{k-1}.
\end{equation}
The largest integer~$n$ such that $\P^n(v,k,\lambda)$-cubes exists
is denoted by $\nu(v,k,\lambda)$. In contrast, the dimension of $\C$-cubes
can be arbitrarily large for fixed parameters $(v,k,\lambda)$.

The organization of our paper is as follows. In Section~\ref{sec2}, we recall the
definitions of isotopy and equivalence of $n$-dimensional incidence cubes.
We discuss the numbers of inequivalent cubes in $\C^n(v,k,\lambda)$ and
$\P^n(v,k,\lambda)$ for $k=1$ and $k=2$. The main result in this section
is Theorem~\ref{class731}, giving a complete enumeration up to equivalence
of $\P$-cubes for parameters $(7,3,1)$, $(7,4,2)$, and all dimensions~$n$.
The proof is a computer calculation based on an algorithm that successively
increases the dimension. The largest possible dimensions are $\nu(7,3,1)=7$
and $\nu(7,4,2)=9$.

Autotopies of $\C$- and $\P$-cubes are studied in Section~\ref{sec3}.
Results about the action of automorphisms of symmetric designs on
the points and blocks carry over to the action of autotopies of
$\P$-cubes on each coordinate. These results do not hold for
autotopies of $\C$-cubes. The main computational result in this
section is Theorem~\ref{v11aut}, giving a complete classification
of $\P^n(11,5,2)$-cubes with nontrivial autotopies. Cubes in
$\P^3(16,6,2)$ with an autotopy of order~$8$ acting semiregularly
are classified in Proposition~\ref{v16aut8}. Among them are examples
with three non-isomorphic $(16,6,2)$ designs as projections,
answering a question posed in~\cite{KR24}.

In Section~\ref{sec4}, we study $\P^n(v,k,\lambda)$-cubes constructed
from higher-dimensional difference sets. Theorem~\ref{dsbound} gives
a sharper bound than~\eqref{dimbound} on the dimension: $n\le v$.
Building on results from~\cite{KR24}, we compute new values of
$\mu_G(v,k,\lambda)$, the largest dimension of $(v,k,\lambda)$ difference
sets in the group~$G$. For elementary abelian groups~$G$, Theorem~\ref{tmelab}
shows that the bound is tight, i.e.\ $\mu_G(v,k,\lambda)=v$ holds
whenever difference sets exist. Theorem~\ref{tmreg}
generalizes the construction to arbitrary groups~$G$ based on the
notion of regular sets of (anti)automorphisms. A nice consequence is
Corollary~\ref{cycdim}, showing that cyclic $(v,k,\lambda)$ difference
sets extend at least to dimension~$p$, where~$p$ is the smallest prime
divisor of~$v$.

In the final Section~\ref{sec5}, we put forward some observations
based on data in Table~\ref{tab4}. The table contains numbers
of inequivalent $\P^n(v,k,\lambda)$-cubes constructed from difference
sets. We have only partial explanations for apparent symmetries of
the numbers and hope that the observations could lead to new theorems.

Some of our results have traditional formal proofs, while others are proved
by computer calculations. There are numerous connections between the two types
of results in both directions. In Section~\ref{sec3}, formal results
about autotopies enable computer classifications of $\P$-cubes with prescribed
autotopy groups. In Section~\ref{sec4}, computer classifications of
higher-dimensional difference sets provide impetus for formal results,
notably Theorems~\ref{tmelab} and~\ref{tmreg}. Higher-dimensional incidence
cubes are most easily explored by studying examples and performing experiments
on a computer. Tools for examining $\C$- and $\P$-cubes of symmetric designs
are available in the GAP~\cite{GAP} package \emph{Prescribed automorphism
groups}~\cite{PAG}. Plenty of examples are given in this paper,
as well as in~\cite{KPT24, KR24}.

\section{Equivalence and classification}\label{sec2}

Formally, incidence matrices of symmetric $(v,k,\lambda)$ designs
are functions $A:\{1,\ldots,v\}\times \{1,\ldots,v\} \to \{0,1\}$
satisfying equation~\eqref{eqsbibd}. An important consequence of
the equation is non-singularity.

\begin{lemma}\label{nonsing}
Incidence matrices of symmetric $(v,k,\lambda)$ designs are invertible.
\end{lemma}

The lemma implies that the transposed matrix~$A^t$ is also
an incidence matrix of a $(v,k,\lambda)$ design, called the \emph{dual
design}. We refer to the monographs~\cite{IS06, EL83} and the
book~\cite{BJL99} for proofs. Symmetric designs with incidence
matrices~$A$ and~$A'$ are \emph{isomorphic} if there are permutations
$\pi_1, \pi_2\in S_v$ such that $A'(i,j)=A(\pi_1^{-1}(i),\pi_2^{-1}(j))$,
$\forall i,j\in \{1,\ldots,v\}$. This can be written in matrix form as
$A'=P_1AP_2^t$, where $P_1$ and $P_2$ are permutation matrices corresponding
to~$\pi_1$ and~$\pi_2$. We call~$A$ and~$A'$ \emph{equivalent} if~$A'$ is
isomorphic to $A$ or $A^t$. The equivalence classes are orbits of the wreath
product $S_v\wr S_2$ acting on the set of all incidence matrices and
represent symmetric designs up to isomorphism and duality.

Both $\C$- and $\P$-cubes are defined as $n$-dimensional incidence matrices
of order~$v$, i.e.\ as functions $C:\{1,\ldots,v\}^n \to \{0,1\}$ with the
Cartesian $n$-ary power of $\{1,\ldots,v\}$ as domain. Isomorphism of
$n$-dimensional matrices is called \emph{isotopy}: $C$ and $C'$ are isotopic
if there are permutations $\pi_1,\dots,\pi_n\in S_v$ such that
$C'(i_1,\ldots,i_n)=C(\pi_1^{-1}(i_1),\ldots,\pi_n^{-1}(i_n))$,
$\forall i_1,\ldots,i_n\in \{1,\ldots,v\}$. Now the order of the coordinates
can be permuted by any $\gamma \in S_n$. This is called \emph{conjugation}:
$$C^\gamma (i_1,\ldots,i_n) = C(i_{\gamma^{-1}(1)},\ldots,i_{\gamma^{-1}(n)}),\kern 2mm
\forall i_1,\ldots,i_n \in \{1,\ldots,v\}.$$
The cubes $C$ and $C'$ are \emph{equivalent} or \emph{paratopic} if $C'$ is
isotopic to a conjugate $C^\gamma$. The equivalence classes are orbits of the
wreath product $S_v\wr S_n$ acting on $\C^n(v,k,\lambda)$ or $\P^n(v,k,\lambda)$.
The terminology is borrowed from latin squares~\cite{KD15}, where equivalence classes
of paratopy are usually called \emph{main classes}.

The classification problem is to determine the equivalence classes of
$\C$- and $\P$-cubes for given parameters $(v,k,\lambda)$ and dimension~$n$.
We start with the degenerate case $k=1$. Incidence matrices of symmetric
$(v,1,0)$ designs are permutation matrices of order~$v$. They are all
equivalent, since the rows and columns can be permuted to get the identity
matrix~$I$. This extends straightforwardly to $\P^n(v,1,0)$-cubes, which
are all equivalent to
$$C(i_1,\ldots,i_n)=\left\{\begin{array}{ll} 1, & \mbox{ if } i_1=\ldots=i_n,\\[1mm]
0, & \mbox{otherwise.}\\ \end{array}\right.$$

On the other hand, the classification of $\C^n(v,1,0)$-cubes is a difficult problem.
For $n=3$, they are in $1$-to-$1$ correspondence with latin squares $L=(\ell_{i_1i_2})$
of order~$v$ by $C(i_1,i_2,i_3)=[\ell_{i_1 i_2}=i_3]$. The square bracket $[P]$ is
the \emph{Iverson symbol}, taking the value~$1$ if $P$ is true, and~$0$
otherwise~\cite{DK92}. In Figure~\ref{fig0b}, representatives of the two main
classes of latin squares of order~$4$ and the corresponding $\C^3(4,1,0)$-cubes
are shown. The number of main classes of latin squares has been
determined by computer calculations up to $v=11$~\cite{HKO11, MMM07}.

\begin{figure}[t]
\begin{center}
\includegraphics[width=100mm]{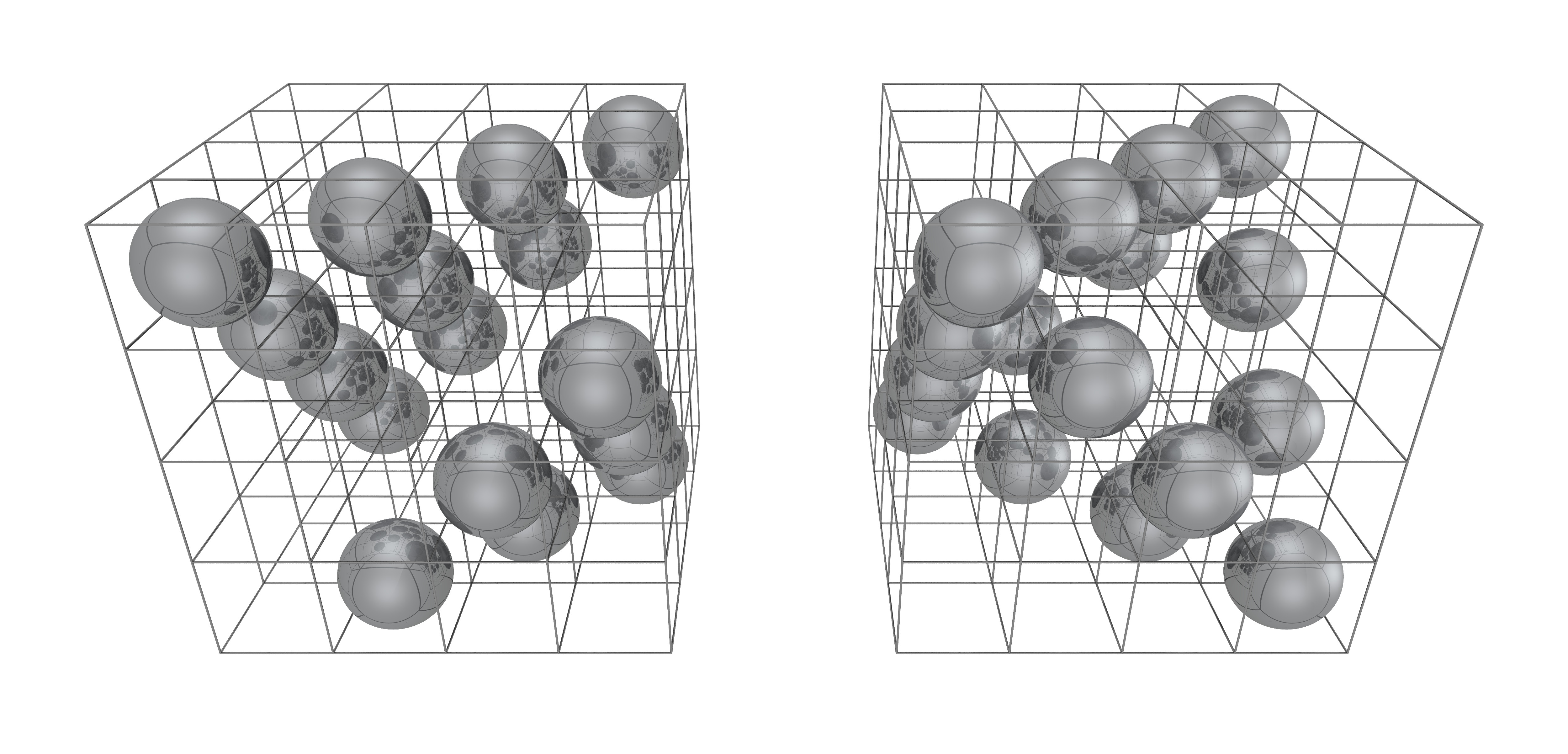}
\end{center}
$$\left(
\begin{array}{llll}
 1 & 2 & 3 & 4 \\
 2 & 3 & 4 & 1 \\
 3 & 4 & 1 & 2 \\
 4 & 1 & 2 & 3
\end{array}
\right)\kern 12mm
\left(
\begin{array}{llll}
 1 & 2 & 3 & 4 \\
 2 & 1 & 4 & 3 \\
 3 & 4 & 1 & 2 \\
 4 & 3 & 2 & 1
\end{array}
\right)$$
\caption{Two $\C^3(4,1,0)$-cubes and the corresponding latin squares of order~$4$.}\label{fig0b}
\end{figure}

Cubes in $\C^n(v,1,0)$ of arbitrary dimension~$n$ are in $1$-to-$1$ correspondence with
latin hypercubes of order~$v$ and dimension~$n-1$. There have been several definitions
of latin hypercubes in the literature; the suitable definition for our purpose is
used in~\cite{MW08}. A \emph{latin hypercube} of order~$v$ and dimension~$n$ is a
$v\times\cdots\times v$ matrix with entries from $\{1,\ldots,v\}$, such that
every $1$-section contains every symbol exactly once. The set of all such objects
is denoted by $\H^n_v$. In~\cite{MW08}, classification is performed for $n=3$, $v\le 6$
and $n=4,5$, $v\le 5$.

For $k=2$, the only feasible parameters of symmetric designs are $(3,2,1)$.
The number of $\C^n(3,2,1)$ cubes can be determined by complementation, i.e.\
by exchanging $0\leftrightarrow 1$. This is a bijection between $\C^n(v,k,\lambda)$
and $\C^n(v,v-k,v-2k+\lambda)$.

\begin{proposition}
The total number of $\C^n(3,2,1)$-cubes is $3\cdot 2^{n-1}$ and they
are all isotopic.
\end{proposition}

\begin{proof}
By complementation, it suffices to count cubes in $\C^n(3,1,0)$.
They are in $1$-to-$1$ correspondence with latin hypercubes of
order $v=3$ and dimension $n-1$. The number of such hypercubes
is known to be $|\H_3^{n-1}|=3\cdot 2^{n-1}$, see~\cite[page 726]{MW08}
or~\cite[Theorem~1]{ES06}. The proof in~\cite{MW08} relies on the
fact that all latin hypercubes of order $v\le 3$ are linear. Linear
hypercubes are easy to count and they are all isotopic. Therefore,
the cubes in $\C^n(3,2,1)$ are also all isotopic.
\end{proof}

\begin{proposition}\label{nocompl}
The complement of a $\P$-cube of dimension $n\ge 3$ is not a $\P$-cube.
\end{proposition}

\begin{proof}
By~\cite[Proposition~2.3]{KR24}, $\P^n(v,k,\lambda)$-cubes have $vk$
incidences ($1$-entries) regardless of the dimension~$n$. The number
of incidences in the complement is $v^n-vk=v(v^{n-1}-k)$. This is too
large for any $\P^n(v,k',\lambda')$-cube if $n>2$.
\end{proof}

Cubes in $\P^n(3,2,1)$ have been classified directly in~\cite{KR24}.
The result is reproduced in the first row of Table~\ref{tab1}. We
see that $\nu(3,2,1)=5$ and the bound~\eqref{dimbound} is tight in
this case. The next feasible parameters are $(7,3,1)$. It is well
known that the Fano plane is the unique $(7,3,1)$ design. In~\cite{KR24},
two inequivalent $\P^3(7,3,1)$-cubes are presented as Examples~2.2 and~2.6.
Here we present a full classification of $\P^n(7,3,1)$ and
$\P^n(7,4,2)$-cubes.

\begin{theorem}\label{class731}
The numbers of cubes in $\P^n(7,3,1)$ and $\P^n(7,4,2)$ up to equivalence
are given in Table~\textup{\ref{tab1}}. In particular, $\nu(7,3,1)=7$ and
$\nu(7,4,2)=9$.
\end{theorem}

\begin{table}[!h]
\begin{tabular}{|c|ccccccccc|}
\hline
 & \multicolumn{9}{c|}{$n$} \\[1mm]
$(v,k,\lambda)$ & 2 & 3 & 4 & 5 & 6 & 7 & 8 & 9 & 10 \\
\hline
\rule{0mm}{5mm}$(3,2,1)$ & 1 & 2 & 1 & 1 & 0 & & & & \\[1mm]
$(7,3,1)$ & 1 & 13 & 20 & 4 & 3 & 2 & 0 & 0 & 0 \\[1mm]
$(7,4,2)$ & 1 & 877 & 884 & 74 & 19 & 9 & 6 & 5 & 0 \\[1mm]
\hline
\end{tabular}
\vskip 2mm
\caption{Numbers of $\P^n(v,k,\lambda)$-cubes up to equivalence.}\label{tab1}
\end{table}

The bound~\eqref{dimbound} gives $\nu(7,3,1)\le 10$ and is not tight for
these parameters, but it is tight for $(7,4,2)$. The proof of Theorem~\ref{class731}
is a computer calculation that was carried out using the
orthogonal array representation of $\P$-cubes. An \emph{orthogonal array}
$OA(N,n,v,t)$ of \emph{size}~$N$, \emph{degree}~$n$, \emph{order}~$v$, and
\emph{strength}~$t$ is an $N$-set of $n$-tuples from $\{1,\ldots,v\}^n$ such
that for any choice of~$t$ coordinates, each $t$-tuple from
$\{1,\ldots,v\}^t$ appears as a restriction of the $n$-tuples to the chosen
coordinates exactly~$\lambda$ times. This parameter is called the \emph{index}
of the orthogonal array and is given by $\lambda=N/v^t$.

A function $C:\{1,\ldots,v\}^n\to \{0,1\}$ is the characteristic function
of a set of $n$-tuples
$$\overline{C}=\{(i_1,\ldots,i_n)\in \{1,\ldots,v\}^n \mid C(i_1,\ldots,i_n)=1\}.$$
If $C\in \P^n(v,k,\lambda)$, then $\overline{C}$ is an orthogonal
array $OA(vk,n,v,1)$ of index~$k$ \cite[Corollary~2.5]{KR24}.
This means that each element from $\{1,\ldots,v\}$ appears
exactly~$k$ times in every coordinate. This property is not sufficient;
a characterization of $OA(vk,n,v,1)$'s representing $\P^n(v,k,\lambda)$-cubes
is given in \cite[Proposition~2.4]{KR24}. In the proof of
\cite[Theorem~2.9]{KR24} it was also shown that $\overline{C}$ is
a distance-invariant non-linear $v$-ary code of length~$n$, in which
only the distances $\{n-1, n\}$ occur.

The OA-representation of $\P$-cubes is convenient
for classification because the size $N=vk$ remains constant
across all dimensions~$n$. We can take the $21$ incident pairs
of the Fano plane and extend them to triples representing
$\P^3(7,3,1)$-cubes, and continue increasing the dimension
by~$1$ in this way. If the conditions from \cite[Proposition~2.4]{KR24}
are not taken into account, the number of extensions of an
$OA(vk,n,v,1)$ is $(vk)!/(k!)^v$. For parameters $(7,3,1)$ this
is already too large for an exhaustive computer search, even if we do
check the conditions for partial extensions using a backtracking algorithm.
We amend this by performing isomorph rejection. As an intermediate step,
we add the entry~$1$ to the new coordinate in ${vk\choose k}$ ways and
check the necessary and sufficient conditions. We then eliminate
equivalent copies among the valid partial extensions, and proceed
by adding the next entry~$2$ in the same way. When we reach the
last entry~$v$, we have all OA-representations of
$\P^{n+1}(v,k,\lambda)$-cubes up to equivalence.

We performed this calculation for parameters $(7,3,1)$ and $(7,4,2)$,
increasing the dimension until further extension is not possible.
We relied on canonical labelings produced by nauty and Traces~\cite{MP14}
to eliminate equivalent (partial) OA-representations of cubes.
Online versions of Table~\ref{tab1} and other tables in this paper are
available on the web page
\begin{center}
\url{https://web.math.pmf.unizg.hr/~krcko/results/pcubes.html}
\end{center}
The online tables contain links to files with OA-representations of
the corresponding cubes that can be used to verify our calculations.
The files are given in GAP~\cite{GAP} format. The GAP package \emph{Prescribed
automorphism groups}~\cite{PAG} contains functions for working with $\P$-
and $\C$-cubes; see the package manual.

We were not able to perform a full classification for the next parameters
$(11,5,2)$, but this might be within reach of today's computer technology.
Instead, we classify $\P^n(11,5,2)$-cubes with nontrivial autotopy groups
in Theorem~\ref{v11aut}. Computer classification of $\C$-cubes is a much more
difficult problem. $\C^n(v,k,\lambda)$-cubes can also be represented as
orthogonal arrays $OA(kv^{n-1},n,v,n-1)$ of index $k$, see~\cite[Section~2]{KPT24}.
However, the size $N=kv^{n-1}$ grows exponentially with the dimension~$n$,
so another approach is needed. We have not found a feasible classification
strategy even for $\C^3(7,3,1)$.

\section{Autotopies}\label{sec3}

Isotopy from an incidence cube to itself is called \emph{autotopy}. The
set of all autotopies of~$C$ forms a group with coordinatewise composition,
called the \emph{full autotopy group} and denoted by $\Atop(C)$. This is a
generalization of the full automorphism group of a symmetric design.
Automorphisms of symmetric designs are usually defined as single
permutations of points, because the corresponding permutations of blocks
are uniquely determined. This is a consequence of Lemma~\ref{nonsing}:
if $A=P_1AP_2^t$ holds for permutation matrices~$P_1$ and~$P_2$,
then $P_2=A^{-1}P_1A$. The next two propositions describe how this
carries over to higher-dimensional $\P$- and $\C$-cubes.

\begin{proposition}
Let $(\pi_1,\ldots,\pi_n)$ be an autotopy of $C\in \P^n(v,k,\lambda)$.
Then any component $\pi_x$ uniquely determines all other components.
\end{proposition}

\begin{proof}
For any other component $\pi_y$, the pair $(\pi_x,\pi_y)$ is an
automorphism of the symmetric design $\Pi_{xy}(C)$. Therefore,
$\pi_y$ is uniquely determined by~$\pi_x$ and $C$.
\end{proof}

\begin{proposition}
Let $(\pi_1,\ldots,\pi_n)$ be an autotopy of $C\in \C^n(v,k,\lambda)$.
Then any component $\pi_x$ is uniquely determined by the $n-1$ other
components.
\end{proposition}

\begin{proof}
Suppose we want to determine $\pi_n$ from $\pi_1,\ldots,\pi_{n-1}$
and~$C$. Let~$A'$ and~$A$ be the $2$-sections of~$C$ obtained by fixing the
first $n-2$ coordinates to $1,\ldots,1$ and $\pi_1^{-1}(1),\ldots,\pi_{n-2}^{-1}(1)$,
respectively. According to the definition, they are incidence
matrices of $(v,k,\lambda)$ designs. The pair $(\pi_{n-1},\pi_{n})$
is an isomorphism between~$A$ and~$A'$:
\begin{align*}
A'(i,j) &= C(1,\ldots,1,i,j)= C(\pi_1^{-1}(1),\ldots,\pi_{n-2}^{-1}(1),\pi_{n-1}^{-1}(i),\pi_n^{-1}(j))\\[1mm]
 &= A(\pi_{n-1}^{-1}(i),\pi_n^{-1}(j)), \kern 2mm \forall i,j\in\{1,\ldots,v\}.
\end{align*}
We can write this in matrix form as $A'=P_{n-1}AP_n^t$. Thus, $\pi_n$ is
uniquely determined by $P_n=(A')^{-1}P_{n-1}A$.
\end{proof}

The previous proposition is best possible, in the sense that~$\pi_x$
is not determined by fewer than $n-1$ other components of the autotopy.
By Theorem~3.4 of~\cite{KPT24}, a $\C$-cube of dimension~$n$
constructed from a $(v,k,\lambda)$ difference set in~$G$ has
an autotopy group isomorphic to $G^{n-1}$. In this group there are~$v$
different choices for $\pi_{n-1}$ and $\pi_n$ for any given components
$\pi_1,\ldots,\pi_{n-2}$.

It is known that automorphisms of symmetric $(v,k,\lambda)$ designs
fix as many points as blocks. The analogous statement is true for
$\P$-cubes.

\begin{proposition}\label{fixp}
Let $\pi\in \Atop(C)$ be an autotopy of $C\in \P^n(v,k,\lambda)$.
Then every component $\pi_x$ has the same number of fixed points.
\end{proposition}

\begin{proof}
The claim follows from the observation that $(\pi_x,\pi_y)$ is an
automorphism of the symmetric design $A=\Pi_{xy}(C)$, for all $1\le x < y\le n$.
If $P_x$ and $P_y$ are the corresponding permutation matrices, then $A=P_x A P_y^t$
holds. By Lemma~\ref{nonsing} we can write this as $P_y=A^{-1}P_x A$.
Now $P_x$ and $P_y$ have the same trace, and this is the number of
fixed points of $\pi_x$ and $\pi_y$.
\end{proof}

As a consequence, autotopy groups of $\P$-cubes have the same number
of orbits on each coordinate.

\begin{proposition}
Let $G\le \Atop(C)$ be an autotopy group of $C\in \P^n(v,k,\lambda)$.
For any $x\in \{1,\ldots,n\}$, $G_x=\{\pi_x \mid \pi\in G\}$ is a permutation
group acting on $\{1,\ldots,v\}$; the number of orbits of $G_x$ does not
depend on the choice of~$x$.
\end{proposition}

\begin{proof}
By the Burnside--Cauchy--Frobenius lemma, the number of orbits can be expressed as
$$\frac{1}{|G_x|} \sum_{\pi_x \in G_x} f(\pi_x) = \frac{1}{|G|} \sum_{\pi \in G} f(\pi_x).$$
Here $f(\pi_x)$ is the number of fixed points and does not depend
on~$x$ by Proposition~\ref{fixp}.
\end{proof}

Now it is clear that an autotopy group $G\le \Atop(C)$ of a $\P$-cube~$C$ acts
(sharply) transitively on one coordinate if and only if it acts (sharply) transitively
on every coordinate; cf.\ \cite[Proposition~3.7]{KR24}. Furthermore, bounds on
the number of fixed points of automorphisms of symmetric $(v,k,\lambda)$
designs apply to all components $\pi_x$ of autotopies of
$\P^n(v,k,\lambda)$-cubes. For example, $f(\pi_x)\le v/2$ by
\cite[Theorem~3]{WF70} and $f(\pi_x)\le k-\sqrt{k-\lambda}$ by
\cite[Corollary~2.5.6]{SBW83}.

On the other hand, autotopies of $\C$-cubes can have components with
different numbers of fixed points. The $\C^3(7,3,1)$ cube
of~\cite[Example~2.2]{KPT24} has an autotopy of order~$7$ with two
components fixing no points, and the third component fixing~$7$ points.
In~\cite[Propositions~5.1 and 5.3]{KPT24}
a total of $2396$ inequivalent $\C^3(16,6,2)$-cubes have been constructed.
They have autotopies with two components fixing no points, and the third
component fixing $f=2$, $4$, $6$, $8$, $12$, $14$, or $16$ points.
This also provides counterexamples for bounds on the number of fixed
points and other claims stated above.

Our next goal is to classify $\P^n(11,5,2)$-cubes with nontrivial autotopy
groups. It suffices to consider autotopies of prime orders~$p$.
By~\cite[Theorem~2.7]{MA71}, either $p$ divides $v$ or $p\le k$.
The unique $(11,5,2)$ design has full automorphism group of order $660$,
hence all possible orders $p\in \{2,3,5,11\}$ occur for dimension $n=2$.
Cubes with autotopies of order~$11$ are obtained from higher-dimensional
$(11,5,2)$ difference sets and have already been classified in~\cite{KR24};
the result is reproduced in the first row of Table~\ref{tab2}. Interestingly,
each one of these cubes has full autotopy group of order~$55$ isomorphic
to $\Z_{11}\rtimes \Z_5$.

For $p=5$ we adopt the classification strategy from the previous section of
successively increasing the dimension. By Proposition~\ref{fixp} we know that
every component of an autotopy of order~$5$ has one fixed point and two cycles
of length~$5$. Crucially, autotopies are preserved when $\P^n(11,5,2)$-cubes
are restricted to $n-1$ dimensions by deleting a coordinate in the
OA-representation. We can therefore extend the $\P^{n-1}(11,5,2)$-cubes
with autotopy of order~$5$, respecting the autotopy on the added coordinate
and the conditions from \cite[Proposition~2.4]{KR24}. If we do this exhaustively,
we get all $\P^n(11,5,2)$-cubes with the prescribed autotopy. The procedure
was implemented in the C programming language and required a modest amount of
CPU time to perform a full classification for $p=5$. Nauty and Traces~\cite{MP14}
were used to eliminate equivalent cubes and to compute full autotopy groups.

\begin{proposition}\label{v11p5}
Up to equivalence, the numbers of $\P^n(11,5,2)$-cubes with an
autotopy of order~$5$ are given in the second row of Table~\textup{\ref{tab2}}.
\end{proposition}

\begin{table}[t]
\begin{tabular}{|c|ccccccccccc|}
\hline
 & \multicolumn{11}{c|}{$n$}\\[1mm]
$p$ & 2 & 3 & 4 & 5 & 6 & 7 & 8 & 9 & 10 & 11 & 12\\
\hline
\rule{0mm}{5mm}$11$ & 1 & 2 & 4 & 6 & 6 & 4 & 2 & 1 & 1 & 1 & 0 \\[1mm]
$5$ & 1 & 283 & 443 & 8 & 7 & 4 & 2 & 1 & 1 & 1 & 0 \\[1mm]
$3$ & 1 & 4758 & 0 & 0 & 0 & 0 & 0 & 0 & 0 & 0 & 0 \\[1mm]
$2$ & 1 & 5142 & 0 & 0 & 0 & 0 & 0 & 0 & 0 & 0 & 0 \\[1mm]
\hline
\rule{0mm}{5mm}Total & 1 & 10178 & 443 & 8 & 7 & 4 & 2 & 1 & 1 & 1 & 0 \\[1mm]
\hline
\end{tabular}
\vskip 2mm
\caption{The $\P^n(11,5,2)$-cubes with nontrivial autotopies.}\label{tab2}
\end{table}

All the new examples from the previous proposition have full autotopy
groups of order~$5$. Of course, the examples with full autotopy groups
of order~$55$ were also recovered. The next case $p=3$ was settled by
a similar calculation.

\begin{proposition}
The number of inequivalent $\P^n(11,5,2)$-cubes with an autotopy of order~$3$
is $4758$ for $n=3$ and $0$ for $n\ge 4$; see the third row of Table~\textup{\ref{tab2}}.
\end{proposition}

\begin{proof}
Automorphisms of order~$3$ of the unique $(11,5,2)$ design have
$2$ fixed points, each incident with $2$ fixed blocks. By
Proposition~\ref{fixp} we can assume that all components of
the autotopy are $\pi_x=(3,4,5)(6,7,8)$ $(9,10,11)$. An exhaustive
computer search, systematically extending the $55$ incident pairs
of the $(11,5,2)$ design, produced $4758$ inequivalent OA-representations
of $\P^3(11,5,2)$-cubes invariant under the prescribed autotopy.
A few hours of CPU time were required. In the next step none of the
$3$-cubes extend to~$4$ dimensions. This can be inferred directly
by considering the fixed elements. For $n=2$, incidences of the fixed
elements are $(1,1)$, $(1,2)$, $(2,1)$, $(2,2)$. The pairs can be
extended to triples $(1,1,1)$, $(1,2,2)$, $(2,1,2)$, $(2,2,1)$,
but not to quadruples of fixed elements satisfying all requirements.
\end{proof}

Five $3$-cubes from the previous proposition have full autotopy
groups of order~$12$ isomorphic to $(\Z_2\times \Z_2)\rtimes \Z_3$.
The $4753$ other $3$-cubes have full autotopy groups of
order~$3$. It remains to classify cubes with autotopies of
order~$p=2$, i.e.\ involutions.

\begin{proposition}\label{v11p2}
Up to equivalence, the number of $\P^n(11,5,2)$-cubes with an
involutory autotopy is $5142$ for $n=3$ and $0$ for $n\ge 4$;
see the fourth row of Table~\textup{\ref{tab2}}.
\end{proposition}

\begin{proof}
Automorphisms of order $p=2$ of the $(11,5,2)$ design have~$3$ fixed
points and blocks, forming $3$ incident pairs. Now incidences of the
fixed elements can be extended to any dimension. The cubes exist for
$n=3$; we classified them by the algorithm described above. The
resulting $5142$ inequivalent $3$-cubes cannot be extended to
dimension $n=4$. This was established by running the algorithm in
parallel and required some $2$ years of CPU time in total. Not a
single extension was found, so $\P^n(11,5,2)$-cubes with an involutory
autotopy do not exist for $n\ge 4$.
\end{proof}

The five $\P^3(11,5,2)$-cubes with full autotopy groups of order~$12$
were also found in Proposition~\ref{v11p2}. Of the remaining cubes,
$71$ have full autotopy groups of order~$4$ isomorphic to $\Z_2\times \Z_2$
and $5066$ have full autotopy groups of order~$2$. We calculated the total
numbers of $\P^n(11,5,2)$-cubes with nontrivial autotopies by concatenating
the lists for $p=11$, $5$, $3$, $2$ and eliminating equivalent copies.

\begin{theorem}\label{v11aut}
Cubes in $\P^n(11,5,2)$ with nontrivial autotopy groups exist
if and only if $2\le n\le 11$. The number of such cubes up to
equivalence is $1$, $10178$, $443$, $8$, $7$, $4$, $2$, $1$, $1$,
and $1$ for successive dimensions~$n$ in this range.
\end{theorem}

In \cite{KPT24} and \cite{KR24}, constructions of $C^3(16,6,2)$ and
$\P^3(16,6,2)$-cubes with prescribed autotopy groups $G$ were performed
by a different approach. Essentially, all orthogonal arrays
$OA(1536,3,16,2)$ and $OA(96,$ $3,16,1)$ invariant under~$G$ were
constructed and the ones not representing $\C$- and $\P$-cubes
were discarded. Because of high proportions of ``bad'' OAs, rather
large prescribed groups had to be used: $|G|\ge 512$ in~\cite{KPT24}
and $|G|\ge 18$ in~\cite{KR24}. We now have a more efficient approach
for $\P$-cubes and can handle smaller groups. As an example, we present
the following result.

\begin{proposition}\label{v16aut8}
There are exactly $1076$ inequivalent $\P^3(16,6,2)$-cubes with an
autotopy of order~$8$ acting in two cycles on each coordinate.
\end{proposition}

\begin{table}[!b]
\begin{tabular}{|c|cccccccccc|c|}
\hline
$(T,P)$ & \rotatebox{90}{\rule{-1.5mm}{0mm}$(\D_R,\D_R,\D_R)$\rule{1mm}{0mm}}
& \rotatebox{90}{\rule{-1.5mm}{0mm}$(\D_R,\D_R,\D_G)$\rule{1mm}{0mm}}
& \rotatebox{90}{\rule{-1.5mm}{0mm}$(\D_R,\D_R,\D_B)$\rule{1mm}{0mm}}
& \rotatebox{90}{\rule{-1.5mm}{0mm}$(\D_R,\D_G,\D_G)$\rule{1mm}{0mm}}
& \rotatebox{90}{\rule{-1.5mm}{0mm}$(\D_R,\D_G,\D_B)$\rule{1mm}{0mm}}
& \rotatebox{90}{\rule{-1.5mm}{0mm}$(\D_R,\D_B,\D_B)$\rule{1mm}{0mm}}
& \rotatebox{90}{\rule{-1.5mm}{0mm}$(\D_G,\D_G,\D_G)$\rule{1mm}{0mm}}
& \rotatebox{90}{\rule{-1.5mm}{0mm}$(\D_G,\D_G,\D_B)$\rule{1mm}{0mm}}
& \rotatebox{90}{\rule{-1.5mm}{0mm}$(\D_G,\D_B,\D_B)$\rule{1mm}{0mm}}
& \rotatebox{90}{\rule{-1.5mm}{0mm}$(\D_B,\D_B,\D_B)$\rule{1mm}{0mm}}
& Total \\[1mm]
\hline
\hline
 \rule{0mm}{5mm}$(8,1)$ & 4 & 48 & 76 & 124 & 152 & 56 & 102 & 136 & 48 & 35 & 781 \\[1mm]
 $(8,2)$ & 21 & 0 & 8 & 72 & 0 & 64 & 0 & 8 & 0 & 6 & 179 \\[1mm]
 $(8,3)$ & 0 & 0 & 0 & 0 & 0 & 0 & 6 & 0 & 0 & 5 & 11 \\[1mm]
 $(8,6)$ & 1 & 0 & 0 & 0 & 0 & 0 & 0 & 0 & 0 & 2 & 3 \\[1mm]
 $(16,1)$ & 23 & 8 & 0 & 16 & 0 & 0 & 12 & 0 & 0 & 0 & 59 \\[1mm]
 $(16,2)$ & 18 & 0 & 0 & 8 & 0 & 0 & 0 & 0 & 0 & 0 & 26 \\[1mm]
 $(16,3)$ & 0 & 0 & 0 & 0 & 0 & 0 & 4 & 0 & 0 & 0 & 4 \\[1mm]
 $(16,6)$ & 4 & 0 & 0 & 0 & 0 & 0 & 0 & 0 & 0 & 0 & 4 \\[1mm]
 $(32,1)$ & 1 & 0 & 0 & 0 & 0 & 0 & 0 & 0 & 0 & 0 & 1 \\[1mm]
 $(32,2)$ & 3 & 0 & 0 & 0 & 0 & 0 & 0 & 0 & 0 & 0 & 3 \\[1mm]
 $(32,3)$ & 1 & 0 & 0 & 0 & 0 & 0 & 0 & 0 & 0 & 0 & 1 \\[1mm]
 $(32,6)$ & 1 & 0 & 0 & 0 & 0 & 0 & 0 & 0 & 0 & 0 & 1 \\[1mm]
 $(48,2)$ & 1 & 0 & 0 & 0 & 0 & 0 & 0 & 0 & 0 & 0 & 1 \\[1mm]
 $(48,6)$ & 1 & 0 & 0 & 0 & 0 & 0 & 0 & 0 & 0 & 0 & 1 \\[1mm]
 $(96,6)$ & 1 & 0 & 0 & 0 & 0 & 0 & 0 & 0 & 0 & 0 & 1 \\[1mm]
\hline
\rule{0mm}{5mm}Total & 80 & 56 & 84 & 220 & 152 & 120 & 124 & 144 & 48 & 48 & 1076 \\[1mm]
\hline
\end{tabular}
\vskip 2mm
\caption{$\P^3(16,6,2)$-cubes with an autotopy of order~$8$ acting semiregularly.
Table of the statistics of the $1076$ cubes in Proposition~\ref{v16aut8}.}\label{tab3}
\end{table}

\begin{proof}
Three $(16,6,2)$ designs exist~\cite{MR07}. We will call them the \emph{red},
\emph{green} and \emph{blue} design and denote them by $\D_R$, $\D_G$, and $\D_B$.
The full automorphism groups are of orders $|\Aut(\D_R)| = 11520$,
$|\Aut(\D_G)| = 768$, and $|\Aut(\D_B)| = 384$. All three designs have
automorphisms of order~$8$ acting in two cycles of length~$8$ on the points
and blocks. We ran the dimension increasing algorithm and found that they
extend to $440$, $744$, and $596$ inequivalent $\P^3(16,6,2)$-cubes invariant
under the prescribed autotopy, respectively. The total number of extensions
was determined by concatenating the lists and eliminating equivalent cubes.
\end{proof}

\begin{figure}[!b]
\begin{center}
\includegraphics[width=127mm]{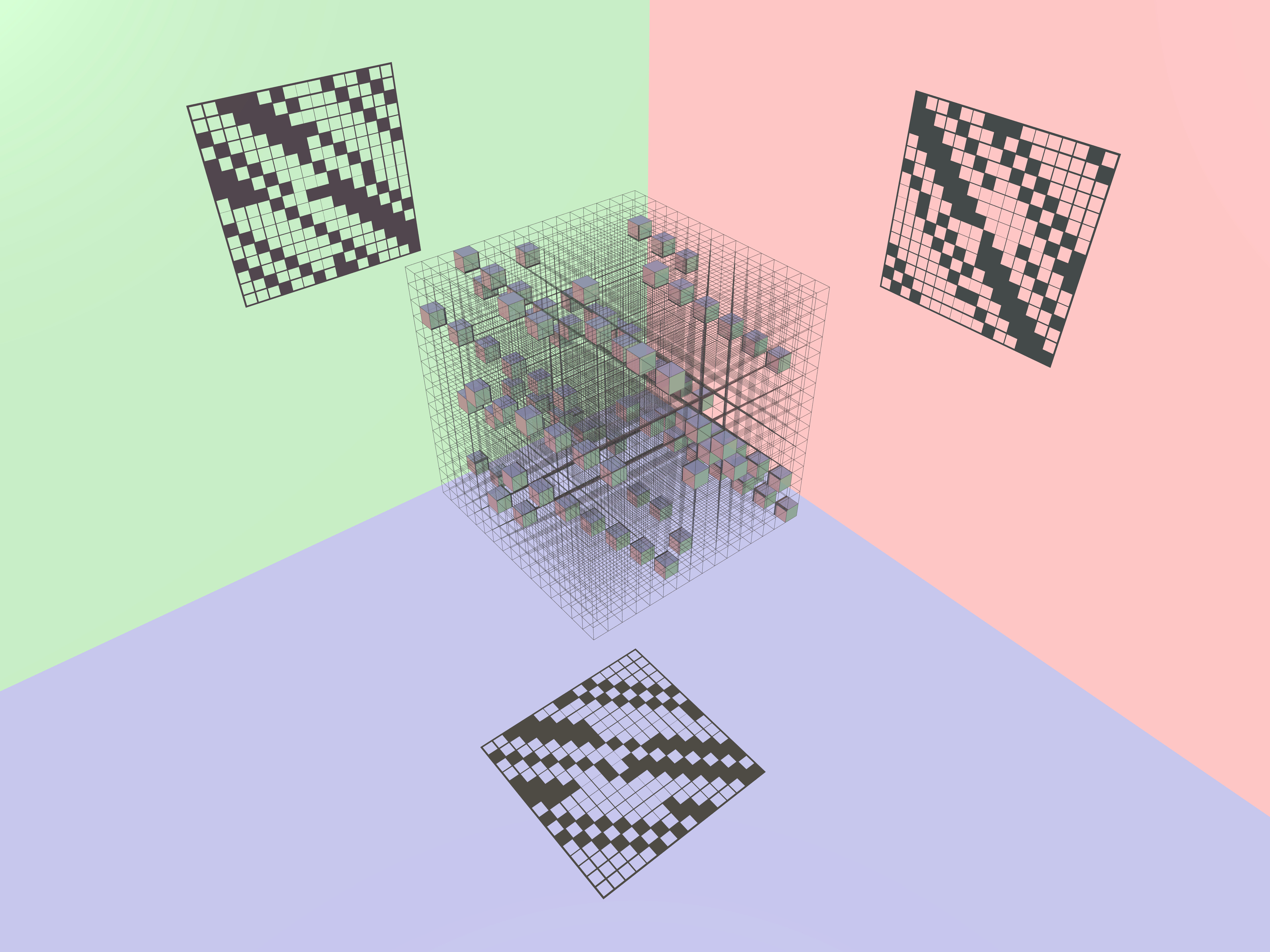}
\end{center}
\caption{A $\P^3(16,6,2)$-cube with non-isomorphic projections.}\label{fig1}
\end{figure}

A detailed statistics of the $1076$ cubes is given in Table~\ref{tab3}.
The cubes are divided according to the size of the full auto(para)topy group and
the $2$-projections. The \emph{full autoparatopy group} $\Apar(C)$ contains all combinations
of isotopy and conjugation mapping~$C$ onto itself. The group size is
given as $(T,P)$, where $T=|\Atop(C)|$ and $P=|\Apar(C)| / T$. A question
whether $\P^3(16,6,2)$-cubes with three non-isomorphic projections exist was
raised in~\cite{KR24}. From the table, we see that there are $152$ cubes with
projections $(\D_R,\D_G,\D_B)$ and the prescribed autotopy of order~$8$. One
example is shown in Figure~\ref{fig1}, with red, green, and blue light sources
so that the projections are in the appropriate color.

\section{Difference sets}\label{sec4}

Let $G$ be an additively written group of order~$v$, not necessarily
abelian. Henceforth we will index cubes with the elements
of~$G$ instead of the integers $\{1,\ldots,v\}$. Thus, $\C$- and
$\P$-cubes are functions $C:G^n\to \{0,1\}$ with all $2$-sections or
$2$-projections being incidence matrices of symmetric
$(v,k,\lambda)$ designs.

A \emph{$(v,k,\lambda)$ difference set} in~$G$ is a $k$-subset
$D\subseteq G$ such that every element $g\in G\setminus\{0\}$ can be
written as $g=a-b$ with $a,b\in D$ in exactly~$\lambda$ ways. By
\cite[Theorem~3.1]{KPT24}, difference sets give rise to $\C$-cubes
of arbitrary dimension $n$. Using the Iverson bracket, a cube
$C\in \C^n(v,k,\lambda)$ is given by
$$C(g_1,\ldots,g_n)=[g_1+\ldots+g_n\in D].$$
This is a special case of \cite[Theorem~2.9]{dL90}. Properties of
these \emph{difference cubes} were studied in Section~3 of~\cite{KPT24}.
In Section~4, the construction was generalized to the so-called
\emph{group cubes} \cite[Theorem~4.1]{KPT24}. This construction
gives examples not equivalent to any difference cube, having
non-isomorphic $(v,k,\lambda)$ designs as $2$-sections. The construction
of \cite[Theorem~4.1]{KPT24} still requires difference sets; the only
known examples of $\C$-cubes not coming from this constructions are
the $\C^3(16,6,2)$-cubes of \cite[Proposition~5.3]{KPT24}.

A stronger kind of difference sets are needed for $\P$-cubes.
An \emph{$n$-dimensional $(v,k,\lambda)$ difference set}
\cite[Definition~3.1]{KR24} is a $k$-subset of $n$-tuples $D\subseteq G^n$
such that $\{d_x-d_y \mid d\in D\}$ is an ``ordinary'' $(v,k,\lambda)$ difference
set for every pair of coordinates $1\le x<y\le n$.
By~\cite[Proposition~3.2]{KR24}, the development
$$\dev D = \{(d_1+g,\ldots,d_n+g) \mid g\in G,\, d\in D\}\subseteq G^n$$
is an OA-representation~$\overline{C}$ of a cube $C\in \P^n(v,k,\lambda)$
indexed by~$G$. Properties of these cubes were studied in Section~3
of~\cite{KR24}. There, it was shown that $n$-dimensional difference sets
can be normalized so that all $n$-tuples in~$D$ start with a $0$
coordinate. Using a different kind of normalization, we will now prove a
bound on the dimension of~$D$ stronger than the bound~\eqref{dimbound}.

\begin{theorem}\label{dsbound}
If an $n$-dimensional $(v,k,\lambda)$ difference set~$D\subseteq G^n$ exists,
then $n\le v$.
\end{theorem}

\begin{proof}
For a group element $g\in G$ and an $n$-tuple $d=(d_1,\ldots,d_n)\in G^n$,
denote by $g+_x d=(d_1,\ldots,d_{x-1},g+d_x,d_{x+1},\ldots,d_n)$. We claim
that $D'=\{g+_x d\mid d\in D\}$ is also an $n$-dimensional $(v,k,\lambda)$
difference set. For any other index $y\in \{1,\ldots,n\}$, the
set $\{d_x-d_y \mid d\in D'\}=\{g+d_x-d_y \mid d\in D\}$ is a left translate
of the difference set $\{d_x-d_y\mid d\in D\}$, hence also a $(v,k,\lambda)$
difference set. By repeated application of this transformation, we can make
an $n$-dimensional difference set~$D'$ containing the $n$-tuple $(0,\ldots,0)$.
Now any other $n$-tuple $d\in D'$ must have distinct coordinates, because
$d_x=d_y$ would imply that the set of differences $\{d_x-d_y \mid d\in D'\}$
contains fewer than~$k$ elements of~$G$. Therefore, the number of
coordinates~$n$ cannot exceed $v=|G|$.
\end{proof}

The normalization of $n$-dimensional difference sets from~\cite{KR24}
does not change the development~$\dev D$. The new normalization
can change the development, but $\dev D$ and $\dev D'$ are always isotopic.
In~\cite[Propositions~3.5 and~3.7]{KR24}, $\P$-cubes coming from difference
sets were characterized as having an autotopy group acting regularly on the
coordinates. Theorem~\ref{class731} shows that the bound $n\le v$ does not
hold for general $\P^n(v,k,\lambda)$-cubes.

The largest integer~$n$ such that an $n$-dimensional $(v,k,\lambda)$ difference
set in~$G$ exists was denoted by $\mu_G(v,k,\lambda)$ in~\cite{KR24}. To determine
values of this function, a computer classification of small $n$-dimensional
difference sets was performed using the library of difference sets available in
the GAP package \emph{DifSets}~\cite{DP19}. In~\cite{KR24}, calculations were
performed in GAP. Exact values of $\mu_G(16,6,2)$ were determined
in~\cite[Table 3]{KR24} for $10$ of the $14$ groups of order~$16$, and lower
bounds were given for the remaining $4$ groups.

We have implemented the classification algorithm in the C programming language
and determined more exact values of~$\mu$ and better lower bounds.
Detailed results of our calculations are presented in Table~\ref{tab4}. The newly
established values of~$\mu$ are summarized in Theorem~\ref{newmuval}. When there
is only one group of order~$v$ up to isomorphism we write $\mu(v,k,\lambda)$, and
when there are several groups we identify them by their ID in the GAP library of
small groups~\cite{GAP}.

\begin{theorem}\label{newmuval}
For groups of order~$16$ with IDs $10$ and $13$ the values of $\mu_G(16,6,2)$
are $4$ and $8$. For groups with IDs $2$, $3$, $4$, $5$, $6$, and $8$ the
values of $\mu_G(16,10,6)$ are $6$, $6$, $6$, $5$, $4$, and $6$,
respectively. Furthermore, $\mu(13,9,6)=13$, $\mu(19,9,4)=\mu(19,10,5)=19$,
$\mu(23,11,5)=\mu(23,12,6)=23$, and $\mu(31,6,1)=31$.
\end{theorem}

\begin{landscape}

\begin{table}[!h]
\vskip 9mm
\begin{tabular}{|cc|cccccccccccccccccc|}
\hline
 & & \multicolumn{18}{c|}{$n$} \\[1mm]
$G$ & $(v,k,\lambda)$ & 2 & 3 & 4 & 5 & 6 & 7 & 8 & 9 & 10 & 11 & 12 & 13 & 14 & 15 & 16 & 17 & 18 & 19 \\[1mm]
\hline
\hline
\multirow{2}{*}{$\Z_7$} & \rule{0mm}{4.5mm}$(7,3,1)$ & 1 & 2 & 2 & 1 & 1 & 1 & & & & & & & & & & & & \\[1mm]
 & $(7,4,2)$ & 1 & 2 & 2 & 1 & 1 & 1 & & & & & & & & & & & & \\[0.5mm]
\hline
\multirow{2}{*}{$\Z_{11}$} & \rule{0mm}{4.5mm}$(11,5,2)$ & 1 & 2 & 4 & 6 & 6 & 4 & 2 & 1 & 1 & 1 & & & & & & & & \\[1mm]
 & $(11,6,3)$ & 1 & 2 & 4 & 6 & 6 & 4 & 2 & 1 & 1 & 1 & & & & & & & & \\[0.5mm]
\hline
\multirow{2}{*}{$\Z_{13}$} & \rule{0mm}{4.5mm}$(13,4,1)$ & 1 & 3 & 7 & 10 & 14 & 14 & 10 & 7 & 3 & 1 & 1 & 1 & & & & & & \\[1mm]
 & $(13,9,6)$ & 1 & 146 & 422 & 652 & 305 & 60 & 13 & 8 & 3 & 1 & 1 & 1 & & & & & & \\[0.5mm]
\hline
\multirow{2}{*}{$\Z_{15}$} & \rule{0mm}{4.5mm}$(15,7,3)$ & 1 & 3 & 0 & 0 & 0 & 0 & 0 & 0 & 0 & 0 & 0 & 0 & 0 & 0 & & & & \\[1mm]
 & $(15,8,4)$ & 1 & 6 & 1 & 0 & 0 & 0 & 0 & 0 & 0 & 0 & 0 & 0 & 0 & 0 & & & & \\[0.5mm]
\hline
\multirow{2}{*}{ID2} & \rule{0mm}{4.5mm}$(16,6,2)$ & 1 & 31 & 81 & 0 & 0 & 0 & 0 & 0 & 0 & 0 & 0 & 0 & 0 & 0 & 0 & & & \\[1mm]
 & $(16,10,6)$ & 1 & 2565 & 152314 & 12115 & 36 & 0 & 0 & 0 & 0 & 0 & 0 & 0 & 0 & 0 & 0 & & & \\[0.5mm]
\hline
\multirow{2}{*}{ID3} & \rule{0mm}{4.5mm}$(16,6,2)$ & 1 & 16 & 55 & 0 & 0 & 0 & 0 & 0 & 0 & 0 & 0 & 0 & 0 & 0 & 0 & & & \\[1mm]
 & $(16,10,6)$ & 1 & 6638 & 462880 & 111294 & 196 & 0 & 0 & 0 & 0 & 0 & 0 & 0 & 0 & 0 & 0 & & & \\[0.5mm]
\hline
\multirow{2}{*}{ID4} & \rule{0mm}{4.5mm}$(16,6,2)$ & 1 & 38 & 113 & 0 & 0 & 0 & 0 & 0 & 0 & 0 & 0 & 0 & 0 & 0 & 0 & & & \\[1mm]
 & $(16,10,6)$ & 1 & 6516 & 389060 & 34076 & 53 & 0 & 0 & 0 & 0 & 0 & 0 & 0 & 0 & 0 & 0 & & & \\[0.5mm]
\hline
\end{tabular}
\vskip 5mm
\caption{Numbers of inequivalent $\P^n(v,k,\lambda)$-cubes obtained from difference sets.}\label{tab4}
\end{table}

\addtocounter{table}{-1}
\setlength{\tabcolsep}{3.8pt}

\begin{table}[!h]
\vskip 9mm
\begin{tabular}{|cc|cccccccccccccccccc|}
\hline
 & & \multicolumn{18}{c|}{$n$} \\[1mm]
$G$ & $(v,k,\lambda)$ & 2 & 3 & 4 & 5 & 6 & 7 & 8 & 9 & 10 & 11 & 12 & 13 & 14 & 15 & 16 & 17 & 18 & 19 \\[1mm]
\hline
\hline
\multirow{2}{*}{ID5} & \rule{0mm}{4.5mm}$(16,6,2)$ & 2 & 56 & 140 & 0 & 0 & 0 & 0 & 0 & 0 & 0 & 0 & 0 & 0 & 0 & 0 & & & \\[1mm]
 & $(16,10,6)$ & 2 & 10680 & 323520 & 6874 & 0 & 0 & 0 & 0 & 0 & 0 & 0 & 0 & 0 & 0 & 0 & & & \\[0.5mm]
\hline
\multirow{2}{*}{ID6} & \rule{0mm}{4.5mm}$(16,6,2)$ & 1 & 8 & 6 & 0 & 0 & 0 & 0 & 0 & 0 & 0 & 0 & 0 & 0 & 0 & 0 & & & \\[1mm]
 & $(16,10,6)$ & 1 & 506 & 1192 & 0 & 0 & 0 & 0 & 0 & 0 & 0 & 0 & 0 & 0 & 0 & 0 & & & \\[0.5mm]
\hline
\multirow{2}{*}{ID8} & \rule{0mm}{4.5mm}$(16,6,2)$ & 1 & 18 & 44 & 0 & 0 & 0 & 0 & 0 & 0 & 0 & 0 & 0 & 0 & 0 & 0 & & & \\[1mm]
 & $(16,10,6)$ & 1 & 3746 & 76580 & 5444 & 8 & 0 & 0 & 0 & 0 & 0 & 0 & 0 & 0 & 0 & 0 & & & \\[0.5mm]
\hline
ID9 & \rule{0mm}{4.5mm}$(16,6,2)$ & 1 & 38 & 112 & 0 & 0 & 0 & 0 & 0 & 0 & 0 & 0 & 0 & 0 & 0 & 0 & & & \\[0.5mm]
\hline
ID10 & \rule{0mm}{4.5mm}$(16,6,2)$ & 1 & 86 & 1941 & 0 & 0 & 0 & 0 & 0 & 0 & 0 & 0 & 0 & 0 & 0 & 0 & & & \\[0.5mm]
\hline
ID11 & \rule{0mm}{4.5mm}$(16,6,2)$ & 1 & 24 & 88 & 0 & 0 & 0 & 0 & 0 & 0 & 0 & 0 & 0 & 0 & 0 & 0 & & & \\[0.5mm]
\hline
ID13 & \rule{0mm}{4.5mm}$(16,6,2)$ & 2 & 129 & 4960 & 19734 & 8106 & 374 & 2 & 0 & 0 & 0 & 0 & 0 & 0 & 0 & 0 & & & \\[0.5mm]
\hline
\multirow{2}{*}{$\Z_{19}$} & \rule{0mm}{4.5mm}$(19,9,4)$ & 1 & 8 & 14 & 36 & 86 & 154 & 228 & 280 & 280 & 228 & 154 & 86 & 36 & 14 & 4 & 1 & 1 & 1 \\[1mm]
 & $(19,10,5)$ & 1 & 8 & 14 & 36 & 86 & 154 & 228 & 280 & 280 & 228 & 154 & 86 & 36 & 14 & 4 & 1 & 1 & 1 \\[0.5mm]
\hline
$\Z_{21}$ & \rule{0mm}{4.5mm}$(21,5,1)$ & 1 & 2 & 0 & 0 & 0 & 0 & 0 & 0 & 0 & 0 & 0 & 0 & 0 & 0 & 0 & 0 & 0 & $\cdots$ \\[0.5mm]
\hline
$F_{21}$ & \rule{0mm}{4.5mm}$(21,5,1)$ & 1 & 6 & 0 & 0 & 0 & 0 & 0 & 0 & 0 & 0 & 0 & 0 & 0 & 0 & 0 & 0 & 0 & $\cdots$ \\[0.5mm]
\hline
\end{tabular}
\vskip 5mm
\caption{Numbers of inequivalent $\P^n(v,k,\lambda)$-cubes obtained from difference sets (continued).}
\end{table}

\clearpage

\end{landscape}

\addtocounter{table}{-1}

\begin{table}[!t]
\begin{tabular}{|cc|cccccccccc|}
\hline
 & & \multicolumn{10}{c|}{$n$} \\[1mm]
$G$ & $(v,k,\lambda)$ & 2 & 3 & 4 & 5 & 6 & 7 & 8 & 9 & 10 & 11 \\[1mm]
\hline
\hline
\multirow{2}{*}{$\Z_{23}$} & \rule{0mm}{4.5mm}$(23,11,5)$ & 1 & 11 & 20 & 69 & 207 & 492 & 984 & 1630 & 2282 & 2694 \\[1mm]
 & $(23,12,6)$ & 1 & 11 & 20 & 69 & 207 & 492 & 984 & 1630 & 2282 & 2694 \\[0.5mm]
\hline
$\Z_{31}$ & \rule{0mm}{4.5mm}$(31,6,1)$ & 1 & 10 & 49 & 195 & 812 & \raisebox{0.25pt}{\footnotesize 2846} &
\raisebox{0.25pt}{\footnotesize 8528} & \raisebox{0.5pt}{\scriptsize 21731} &
\raisebox{0.5pt}{\scriptsize 47801} & \raisebox{0.5pt}{\scriptsize 91148} \\[0.5mm]
\hline
\end{tabular}
\vskip 2mm \setlength{\tabcolsep}{4.6pt}
\begin{tabular}{|cc|cccccccc|}
\hline
 & & \multicolumn{8}{c|}{$n$} \\[1mm]
$G$ & $(v,k,\lambda)$ & 12 & 13 & 14 & 15 & 16 & 17 & 18 & 19 \\[1mm]
\hline
\hline
\multirow{2}{*}{$\Z_{23}$} & \rule{0mm}{4.5mm}$(23,11,5)$ & 2694 & 2282 & 1630 & 984 & 492 & 207 & 69 & 20 \\[1mm]
 & $(23,12,6)$ & 2694 & 2282 & 1630 & 984 & 492 & 207 & 69 & 20 \\[0.5mm]
\hline
$\Z_{31}$ & \rule{0mm}{4.5mm}$(31,6,1)$ & \raisebox{0.85pt}{\tiny 151924} & \raisebox{0.85pt}{\tiny 221959} &
\raisebox{0.85pt}{\tiny 285357} & \raisebox{0.85pt}{\tiny 323396} & \raisebox{0.85pt}{\tiny 323396} &
\raisebox{0.85pt}{\tiny 285357} & \raisebox{0.85pt}{\tiny 221959} & \raisebox{0.85pt}{\tiny 151924} \\[0.5mm]
\hline
\end{tabular}
\vskip 2mm \setlength{\tabcolsep}{3.3pt}
\begin{tabular}{|cc|cccccccccccc|}
\hline
 & & \multicolumn{12}{c|}{$n$} \\[1mm]
$G$ & $(v,k,\lambda)$ & 20 & 21 & 22 & 23 & 24 & 25 & 26 & 27 & 28 & 29 & 30 & 31 \\[1mm]
\hline
\hline
\multirow{2}{*}{$\Z_{23}$} & \rule{0mm}{4.5mm}$(23,11,5)$ & 4 & 1 & 1 & 1 & & & & & & & & \\[1mm]
 & $(23,12,6)$ & 4 & 1 & 1 & 1 & & & & & & & & \\[0.5mm]
\hline
$\Z_{31}$ & \rule{0mm}{4.5mm}$(31,6,1)$ & \raisebox{0.5pt}{\scriptsize 91148} & \raisebox{0.5pt}{\scriptsize 47801} &
\raisebox{0.5pt}{\scriptsize 21731} & \raisebox{0.25pt}{\footnotesize 8528} & \raisebox{0.25pt}{\footnotesize 2846} &
811 & 187 & 38 & 6 & 1 & 1 & 1 \\[0.5mm]
\hline
\end{tabular}
\vskip 4mm
\caption{Numbers of inequivalent $\P^n(v,k,\lambda)$-cubes obtained from difference sets (continued).}
\end{table}

The previously known values of~$\mu$ are given in \cite[Tables 2 and 3]{KR24}
and can also be read from Table~\ref{tab4}. The groups of order $16$ with
IDs $1$, $7$, $12$, and $14$ are missing from Table~\ref{tab4}. The former
two groups are the cyclic group $\Z_{16}$ and the dihedral group $D_{16}$,
which do not contain difference sets. For the latter two groups we could not
completely classify $n$-dimensional difference sets due to large numbers
of inequivalent $\P$-cubes arising from them. Example~\ref{exid12} establishes
an improved lower bound $\mu_G(16,6,2)\ge 11$ for the group with ID $12$. The
group with ID $14$ is the elementary abelian group $\Z_2^4$, and will be dealt
with in Corollary~\ref{mu16elab}.

\allowdisplaybreaks

\begin{example}\label{exid12}
The group of order $16$ with GAP ID $12$ is isomorphic to $G=\Z_2\times Q_8$,
where $Q_8$ is the quaternion group. If elements of the factor groups are $\Z_2=\{0,1\}$
and $Q_8=\{1,i,j,k,-1,-i,-j,-k\}$, the following is a $11$-dimensional $(16,6,2)$
difference set in~$G$:
{\footnotesize \begin{align*}
\{\, & ( (0,1), (0,1), (0,1), (0,j), (0,i), (0,j), (0,k), (1,1), (1,i), (0,1), (1,-i) ),\\[0.3mm]
   & ( (0,1), (0,j), (1,1), (0,1), (0,-i), (1,1), (0,j), (0,k), (1,1), (1,i), (1,-j) ),\\[0.3mm]
   & ( (0,1), (0,i), (0,j), (1,1), (1,-j), (0,i), (1,1), (0,j), (0,-1), (1,1), (1,i) ),\\[0.3mm]
   & ( (0,1), (1,1), (0,i), (1,-k), (0,1), (0,-i), (0,-1), (0,-1), (0,j), (1,-i), (0,j) ),\\[0.3mm]
   & ( (0,1), (0,-1), (1,-k), (0,-1), (1,1), (0,1), (0,1), (1,i), (0,k), (1,-j), (0,1) ),\\[0.3mm]
   & ( (0,1), (1,-k), (0,-1), (0,i), (0,j), (1,-j), (1,i), (0,1), (0,1), (0,j), (1,1) )\, \}.
\end{align*}}
\end{example}

We see that equality is reached quite
often in the upper bound $\mu_G(v,k,\lambda)\le v$ of Theorem~\ref{dsbound}.
For parameters $(q,(q-1)/2,(q-3)/4)$, this is explained by the construction
of $q$-dimensional Paley difference sets in the additive group of $\F_q$,
$q\equiv 3 \pmod{4}$; see \cite[Theorem~4.1]{KR24}. It was noted that the same
construction works for cyclotomic difference sets ($4$th and $8$th powers in~$\F_q$
for appropriate orders~$q$), but this does not explain the equality $\mu(13,4,1)=
\mu(13,9,6)=13$. We will now generalize this construction to difference sets with
arbitrary parameters in elementary abelian groups. The key observation in the
proof of the following theorem is that group automorphisms $\varphi:G\to G$
preserve difference sets in~$G$.

\begin{theorem}\label{tmelab}
Let $G$ be an elementary abelian group, i.e.\ the additive group of
a finite field~$\F_q$. Then any $(q,k,\lambda)$ difference set in~$G$
extends to dimension~$q$.
\end{theorem}

\begin{proof}
Let $D=\{d_1,\ldots,d_k\}\subseteq G$ be a $(q,k,\lambda)$ difference set,
$\alpha\in \F_q$ a primitive element, and $w=(0,1,\alpha,\alpha^2,\ldots,\alpha^{q-2}) \in G^q$.
An extension of~$D$ to dimension~$q$ is $\vec D = \{d_1 w,\ldots, d_k w\}$, where each
coordinate of~$w$ is multiplied by an element of~$D$. To prove that $\vec D$ is
indeed a $q$-dimensional $(q,k,\lambda)$ difference set, we have
to check that the differences of any two coordinates are difference sets
in~$G$. The first coordinate of any vector in $\vec D$ is $0$, and if we
subtract the coordinate with $\alpha^y$ in~$w$, we get
$\{-\alpha^y d_1,\ldots,-\alpha^y d_k\}$. This is a $(q,k,\lambda)$
difference set in~$G$ because it is the image of~$D$ by the group
automorphism $\varphi:G\to G$, $\varphi(g)=-\alpha^y g$. Next, we
subtract two non-zero coordinates, say with~$\alpha^x$ and~$\alpha^y$
in~$w$ for some $x<y$. We get the set
$$\{\alpha^x (1-\alpha^{y-x})d_1,\ldots,\alpha^x (1-\alpha^{y-x})d_k\}.$$
This is again the image of~$D$ by a group automorphism $\psi:G\to G$,
$\psi(g)=\alpha^x(1-\alpha^{y-x})g$.
\end{proof}

Two missing values of $\mu_G$ follow from Theorem~\ref{tmelab}.

\begin{corollary}\label{mu16elab}
For $G=\Z_2^4$, $\mu_G(16,6,2)=\mu_G(16,10,6)=16$.
\end{corollary}

A further generalization of Theorem~\ref{tmelab} applies to arbitrary groups~$G$.
An \emph{antiautomorphism} of~$G$ is a bijection $\varphi:G\to G$ such that
$$\varphi(g+h)=\varphi(h)+\varphi(g),\kern 3mm \forall g,h\in G.$$
If~$G$ is abelian, automorphisms and anti\-auto\-morphisms coincide.
Antiautomorphisms of non-abelian groups are the functions $-\varphi$,
where~$\varphi$ is an automorphism. Thus, antiautomorphisms of~$G$ also
map $(v,k,\lambda)$ difference set~$D\subseteq G$ to difference sets
with the same parameters.

\begin{definition}
We say that $R=\{\varphi_1,\ldots,\varphi_{n-1}\}$ is a \emph{regular
set of (anti)automorphisms} of~$G$ if each $\varphi_x:G\to G$ is an
automorphism or antiautomorphism, and each difference $\varphi_x-\varphi_y$
is an automorphism or antiautomorphism for $1\le x<y\le n-1$.
\end{definition}

Given such a set~$R$ of size $n-1$ and an element $d\in D$,
denote by $\vec d$ the $n$-tuple $(0,\varphi_1(d),\ldots,\varphi_{n-1}(d))\in G^n$.
Then the set $\vec D = \{\vec d \mid d\in D\}$ is an $n$-dimensional
$(v,k,\lambda)$ difference set. This proves the following theorem.

\begin{theorem}\label{tmreg}
If a group $G$ allows a regular set of (anti)automorphisms of size~$n-1$,
then any $(v,k,\lambda)$ difference set in~$G$ extends to dimension~$n$.
\end{theorem}

Theorem~\ref{tmelab} can be seen as a special case of Theorem~\ref{tmreg}
by using multiplication in $\F_q$. The functions $\varphi_x:\F_q\to \F_q$,
$\varphi_x(g)=\alpha^x g$ for $x=0,\ldots,q-2$ constitute a regular set of
automorphisms of the additive group, yielding the construction
of Theorem~\ref{tmelab}. Kerdock sets provide more examples of regular sets
in elementary abelian $2$-groups \cite{CS73, DG75, WMK82a, WMK82b, AMK72}
and $3$-groups \cite{NJP76}.

Theorem~\ref{tmreg} also applies to groups that are not
elementary abelian. For example, $\varphi_1(g)=g$ and
$\varphi_2(g)=2g$ are automorphisms of the group~$\Z_{15}$. The
difference $(\varphi_1-\varphi_2)(g)=-g$ is an auto\-morphism
as well. Thus, $\{\varphi_1,\varphi_2\}$ is a regular set and all
$(15,7,3)$ and $(15,8,4)$ difference sets extend at least to
dimension~$3$. In Table~\ref{tab4}, we see that there is a
$4$-dimensional $(15,8,4)$ difference set, but it cannot be
obtained from Theorem~\ref{tmreg} because~$\Z_{15}$ does not
allow regular sets of size~$3$. The next result applies to
difference sets in arbitrary cyclic groups~$\Z_v$.

\begin{corollary}\label{cycdim}
Let $p$ be the smallest prime divisor of $v$. Then any cyclic
$(v,k,\lambda)$ difference set extends at least to dimension~$p$.
\end{corollary}

\begin{proof}
Automorphisms of $\Z_v$ are of the form $\varphi_a(g)=ag$, where
$a\in \Z_v\setminus \{0\}$ is relatively prime to~$v$. The set
$\{\varphi_1,\ldots,\varphi_{p-1}\}$ is a regular set of
size~$p-1$. Hence, any difference set in~$\Z_v$ extends to
dimension~$p$ by Theorem~\ref{tmreg}.
\end{proof}

For a non-abelian example, we turn to groups of order~$27$.
There are five such groups, but only two of them
contain nontrivial difference sets~\cite{DP19}. These are the
elementary abelian group $\Z_3\times\Z_3\times\Z_3$
(GAP group ID~5) and a semidirect product $\Z_9\rtimes \Z_3$
(GAP group ID~4). Difference sets in the
former group extend to dimension~$27$ by Theorem~\ref{tmelab}.
A regular set in the latter group is given in the next example,
using multiplicative notation.

\begin{example}
The group of order $27$ with GAP ID $4$ can be presented as
$G=\langle a, b \mid a^9=b^3=1,\, ba=a^4b\rangle$. Let $\varphi_1(g)=g$
be the identity automorphism of~$G$ and $\varphi_2(g)=g^{-1}$ the
corresponding antiautomorphism. The ``difference''
$\varphi_1(g)(\varphi_2(g))^{-1}=g^2$ is an antiautomorphism,
namely the one corresponding to $\varphi\in \Aut(G)$, $\varphi(a)=a^7$,
$\varphi(b)=b$. Thus, $\{\varphi_1,\varphi_2\}$ is a regular set
of (anti)automorphisms of~$G$.
\end{example}

By Theorem~\ref{tmreg}, all $(27,13,6)$ and $(27,14,7)$ difference
sets in $G=\langle a, b \mid a^9=b^3=1,\, ba=a^4b\rangle$ extend
to dimension~$3$.

\section{Final observations}\label{sec5}

A curious fact visible in Table~\ref{tab4} is equality of the numbers
for some complementary parameters. Ordinary ($2$-dimensional) difference
sets and designs are rarely discussed for both sets of complementary
parameters, because complementation is a bijection. This is also true
for higher-dimensional $\C$-cubes, but not for $\P$-cubes as we have
seen in Proposition~\ref{nocompl}. However, the numbers in Table~\ref{tab4}
coincide for complementary Paley-Hadamard parameters $\P^n(4m-1,2m-1,m-1)$
and $\P^n(4m-1,2m,m)$, when $4m-1=p^s$ is a prime power and $G=(\Z_p)^s$ is
an elementary abelian group. For small parameters, this can be explained
by the following bijection.

Every $(4m-1,2m-1,m-1)$ difference set $D\subseteq G$ can be uniquely extended
to a $(4m-1,2m,m)$ difference set $D\cup\{a\}$ by adding the element
$a=-2\sum_{d\in D} d$. We have checked that this is a bijection between
$(4m-1,2m-1,m-1)$ and $(4m-1,2m,m)$ difference sets in the
groups $\Z_7$, $\Z_{11}$, $\Z_{19}$, and $\Z_{23}$ from Table~\ref{tab4},
and also in $G=(\Z_3)^3$. We don't have a proof for arbitrary
elementary abelian groups. However, if this is indeed a bijection between
ordinary difference sets, then it is easy to establish a $1$-to-$1$
correspondence between $n$-dimensional difference sets. Given an
$n$-dimensional $(4m-1,2m-1,m-1)$ difference set $D\subseteq G^n$,
one simply adds the $n$-tuple $\left(-2\sum_{d\in D} d_1,\ldots,-2\sum_{d\in D} d_n\right)$
to obtain an $n$-dimensional $(4m-1,2m,m)$ difference set.

If a $(4m-1,2m-1,m-1)$ difference set $D\subseteq G$ extends
to a $(4m-1,2m,m)$ difference set $D\cup\{a\}$, then the complement
$D'=G\setminus (D\cup\{a\})$ is also a $(4m-1,2m-1,m-1)$ difference set
in~$G$. Then $(D-a)\cup (D'-a) \cup \{0\}$ is a tiling of~$G$ by two
difference sets, and $D-a$ is a skew Hadamard difference set by
\cite[Theorem~8]{CKZ15}. A long-standing conjecture was that the classical
Paley difference sets are the only skew Hadamard difference sets in elementary
abelian groups, but it has been refuted in~\cite{DY06, DWX07} for
groups of orders $3^m$, $m\ge 5$ odd. It might also be the case that
our bijection holds only for elementary abelian groups of small
orders, but we don't have a counterexample.

Yet another apparent symmetry in Table~\ref{tab4} is equality of
the numbers for ``complementary dimensions'' $n$ and $v-n$. Equality
holds for some small parameters $(v,k,\lambda)$ when the upper bound
of Theorem~\ref{dsbound} is reached: $(7,3,1)$, $(7,4,2)$, $(11,5,2)$,
$(11,6,3)$, and $(13,4,1)$. It breaks for the complementary parameters
$(13,9,6)$, while for $(19,9,4)$, $(19,10,5)$, $(23,11,5)$, and $(23,12,6)$
equality holds except for dimension $n=3$. For $(31,6,1)$, it breaks for
$n=3$, $4$, $5$, $6$ and holds for the remaining dimensions. We don't have
an explanation for this phenomenon.

\end{document}